\documentclass[12pt]{amsart}
\usepackage[all]{xy}
\usepackage{mathrsfs}
\usepackage[margin=1.0in]{geometry}

\usepackage{graphicx}
\usepackage[cyr]{aeguill}\usepackage[francais,english]{babel}
\usepackage{amsmath,amsfonts,amssymb}
\usepackage{color}
\usepackage[utf8]{inputenc}
\usepackage{comment}
\usepackage[shortlabels]{enumitem}
\usepackage{etoolbox}
\usepackage{float}
\usepackage{latexsym}
\usepackage{lipsum}
\usepackage{needspace}

\usepackage{pgf,tikz,pgfplots}
\usetikzlibrary{arrows,shapes}
\usepackage{hyperref}
\usepackage{relsize}

\newtheorem{theoremIntro}{Theorem}[]

\newtheorem{questionIntro}{Question}
\newtheorem{Openproblem}{Open Question}

\newtheorem{theorem}{Theorem}[section]
\newtheorem{lemma}[theorem]{Lemma}
\newtheorem{proposition}[theorem]{Proposition}

\newtheorem{definition}[theorem]{Definition}
\newtheorem{corollary}[theorem]{Corollary}

\theoremstyle{remark}
\newtheorem{remark}[theorem]{Remark}

\numberwithin{equation}{section}

\begin{document}

	\title{Flexibility and movability in Cayley graphs}

\author{Arindam Biswas}
\address{}
\curraddr{}
\email{arindam.biswas@univie.ac.at, arin.math@gmail.com}
\thanks{}

	

	\begin{abstract}
		Let $\mathbf{\Gamma} = (V,E)$ be a (non-trivial) finite graph with $\lambda: E \rightarrow \mathbb{R}_{+}$, an edge labelling of $\mathbf{\Gamma}$. Let $\rho : V\rightarrow \mathbb{R}^{2}$ be a map which preserves the edge labelling. The graph $\mathbf{\Gamma}$ is said to be flexible if there exists an infinite number of such maps (upto equivalence by rigid transformations) and it is said to be movable if there exists an infinite number of injective maps. We study movability of Cayley graphs and construct regular movable graphs of all degrees. Further, we give explicit constructions of ``dense", movable graphs.
		
	\end{abstract}

	\maketitle
	
	\section{Introduction}
	In this article, we are interested in questions of rigidity, flexibility and movability of finite graphs. By definition,
	\begin{definition}[Graph]
			A graph $\mathbf{\Gamma} = (V,E)$ is a tuple where $V$ is an arbitrary set and $E \subseteq V\times V$, called respectively the set of vertices of $\mathbf{\Gamma}$ and the set of edges of $\mathbf{\Gamma}$. 
	\end{definition}
Strictly speaking $E$ is a \textit{multi-set} (when we have multiple edges between the same pair of vertices), but here we shall be dealing with simple graphs, so $E$ is a subset of $V\times V$. A graph is said to be finite if $|V| < +\infty$, undirected if $(u,v) = (v,u),\,\forall (u,v)\in E$ and without loops if $(v,v)\notin E, \,\,\forall v\in V$. From now on, the graphs we shall consider are non-trivial ($|E|> 0$), finite, simple, undirected, connected and without loops. Also, the graphs will be non-bipartite unless explicitly mentioned.
	
	Our aim is to study notions of rigidity, flexibility and movability in finite graphs. Intuitively, given a graph with fixed edge lengths, it is said to be rigid if none of its vertices can be moved without deforming at least one edge length. For the purpose of this article, the movement we shall consider, will be in the Euclidean space $\mathbb{R}^{2}$, but one can study the topic in other spaces as well. 
		
	Formally, we are given a finite graph $\mathbf{\Gamma}$ with $\lambda: E \rightarrow \mathbb{R}_{+}$, an edge labelling of it. Borrowing from the notation in \cite{GSS93}, we define,	
	\begin{definition}[Framework (in $m$-space)]
		 A framework (in $m$-space) is a triple $(V,E,\rho)$ (or sometimes by abuse of notation, the tuple $(\mathbf{\Gamma},\rho)$), where $\mathbf{\Gamma} =(V,E)$ is a graph and $\rho$ is a map (called the realisation map), 
		 $$ \rho:V \rightarrow \mathbb{R}^{m}.$$
	\end{definition}
	The framework $(V,E,\rho)$ is said to be injective if $\rho$ is injective and realisable if $\rho$ preserves the edge labelling $\lambda$, i.e., $$\| \rho(u) - \rho(v)\|_{m} = \lambda((u,v)), \,\forall (u,v)\in E$$
	where $\|x-y\|_{m}$ denotes the euclidean distance between two points $x,y \in \mathbb{R}^{m}$.
	
	Informally, the edge labelling signifies that the edges have positive lengths, while a realisable framework signifies that these edge lengths are preserved when we map the graph into $\mathbb{R}^{m}$ (hence distances between vertices in the original graph are fixed when we look at them in $\mathbb{R}^{m}$). It is clear that we can obtain infinitely many realizations from a given one by euclidean isometries (reflections, rotations and translations). Two realisations $\rho_{1}$ and $\rho_{2}$ are equivalent if there exist some direct euclidean isometry $\tau$ such that $\rho_{1} = \tau \rho_{2}$.  By the number of realisations of a given graph $\mathbf{\Gamma}$ (with edge labelling $\lambda$), we shall mean the number of equivalence classes of these realisation maps preserving the edge labelling $\lambda$. 
	
		\begin{definition}[rigid, flexible, movable]
		A realisable framework $(\mathbf{\Gamma}, \rho)$ in $\mathbb{R}^{m}$ is defined to be 
		\begin{enumerate}
			\item \textbf{Rigid} if the number of realisations of $\mathbf{\Gamma}$ is finite.
			\item \textbf{Flexible} if the number of realisations of $\mathbf{\Gamma}$ is infinite.
			\item \textbf{Movable} if the number of injective realisations of $\mathbf{\Gamma}$ is infinite.
		\end{enumerate}
	\end{definition}
	As pointed out in \cite{GLS17}, a realisation is not required to be injective as non-adjacent vertices can overlap. However, adjacent vertices have to be mapped to different points. This is because the edge lengths are positive by assumption.
	In \cite{GLS19} it was shown by Grasegger, Legersky and Schicho that every graph except the complete graph is movable in $\mathbb{R}^{3}$. We note that if a graph is movable in $\mathbb{R}^{m}$ then it is automatically movable in $\mathbb{R}^{n}$ where $n>m$. Therefore, it makes sense to restrict to the case $m=2$. In this article, we shall be dealing with equivalence classes of realisation maps in the space of realisable (but not necessarily injective) frameworks in $\mathbb{R}^{2}$.

	By abuse of notation, we shall frequently say that a graph is rigid, flexible or movable when we actually mean that the framework in $\mathbb{R}^{2}$ is rigid, flexible or movable respectively. 
	
	Recently, a combinatorial characterisation of flexible frameworks was given by Grasegger, Legersky and Schicho \cite{GLS17}]. Also, in \cite{GLS19}, they gave a necessary (combinatorial) criterion for movability which they showed is also sufficient for small graphs (the number of vertices should be $\leqslant 8$). 
	
	If the number of edges in the graph is small then the graph has a higher chance of being movable. A connected graph on $|V|$ vertices must have at least $|V|-1$ edges. It is moving.  At the other extreme, the complete (simple) graph on $|V|$ vertices with $\frac{|V|(|V|-1)}{2}$ edges is rigid. It simply has too many edges for it be moving. Henceforth, we shall call a sequence $n=1,2,\cdots,$ of finite simple graphs on $V_{n}$ vertices a \textit{dense sequence} of \emph{density $\alpha$} if they have more than $O(|V_{n}|^{\alpha})$ edges. We note that if $|E_{n}|$ denote the number of edges for each graph in the sequence, then $|V_{n}|-1 \leqslant |E_{n}|\leqslant \frac{|V_{n}|(|V_{n}|-1)}{2}$, i.e., $O(|V_{n}|) \leqslant |E_{n}|\leqslant O(|V_{n}|^{2}) \Rightarrow 1\leqslant \alpha \leqslant 2.$\\
	\\
	The questions which we address in this article are -
	\begin{questionIntro}\label{QIntro1}
		Are there a large class of finite graphs which are flexible and/or movable? Can we give general methods of construction of such graphs? How about movable graphs which are also regular\footnote{a graph is said to be $r$-regular (where $r\geqslant 1$ is an integer) if there are exactly $r$ half edges connected to each vertex.}?
	\end{questionIntro}

	We note that, if a finite graph $\mathbf{\Gamma}$ is movable then the subgraph obtained by deleting some edge(s) which form(s) some cycle(s) is also movable while if we add some outer edge (i.e., an edge which increases the number of vertices of $\mathbf{\Gamma}$ by one), then the new graph is also movable. This is a strong motivation to consider movable, regular graphs as deleting some edge(s) (which are part(s) of some cycles) from these graphs  will keep the movability property unchanged (and making the graphs irregular if we so wish).

	\begin{questionIntro}\label{QIntro2}
		 Given $V_{n}$ vertices, a movable graph with $|E_{n}|$ edges must satisfy $|V_{n}|-1\leqslant |E_{n}| < \frac{|V_{n}|(|V_{n}|-1)}{2}$. Can we construct dense sequences (with density $\alpha$) of finite, regular, movable graphs?
	\end{questionIntro}

	\subsection{Statement of Results}
	
	We show the following results -

	\begin{theoremIntro}[Movability of Cayley graphs - Theorem \ref{ThmCay}]\label{thmCay}
	Let $G$ be a finite group and $S$ be a generating set of $G$ without the identity. There exist certain general conditions on $S$, such that the undirected, simple Cayley graph $C(G,S)$ becomes flexible and also movable. 
	\end{theoremIntro}

	 From the above conditions, one can deduce the following corollary:
\begin{corollary}\label{CorA1}
	The Cayley graph $C(G,S)$ with respect to a set of generators $S$ having the property that $\langle s_{i}\rangle \cap \langle s_{j}\rangle = \lbrace e \rbrace \,\forall s_{i},s_{j} \in S$ with $i\neq j$ is always movable..
\end{corollary}

	Also, one can construct explicitly, movable graphs of all regularity in both the families of finite abelian groups and the finite non-abelian groups.
	\begin{theoremIntro}[Theorem \ref{MovRegAb}]
		There exist movable graphs of all regularity coming from abelian groups. Conversely, given any finite abelian group, there exist a symmetric generating set with respect to which its Cayley graph is movable.
	\end{theoremIntro}

For questions of density, we establish the movability of cartesian product of graphs - Proposition \ref{CartesianProdProp} and from there construct the following dense sequences of graphs in abelian and non-abelian groups.

	\begin{theoremIntro}[Dense Cayley graphs of abelian groups - Theorem \ref{DenseMovCay}]\label{ThmC}
	For each $\alpha \in [1,2]$, there exist dense sequences of movable Cayley graphs of abelian groups of density $\alpha$. Thus, asymptotically the densest possible sequence of moving graphs  (corresponding to $\alpha = 2$) can be attained.
	\end{theoremIntro}

\begin{theoremIntro}[Dense Cayley graphs of non-abelian groups - Theorem \ref{DensestMovCay}]\label{ThmD}
	For each $\alpha \in [1,2]$, there exist dense sequences of movable Cayley graphs of non-abelian groups of density $\alpha$. 
\end{theoremIntro}

As a corollary of the method used to show Theorem \ref{ThmD} we get that
\begin{corollary}\label{CorA2}
		There exist movable graphs of all regularity coming from non-abelian groups.
\end{corollary}

Finally, in the concluding section, we point out some open questions and further directions of research.
\subsection{Acknowledgements}
The author is grateful to Josef Schicho for a number of helpful discussions on rigidity and flexibility of graphs and for his encouragement in pursuing this work. The work was initiated while on a visit to the Johann Radon Institute for Computational and Applied Mathematics (RICAM) and the Johannes Kepler University (JKU) Linz. The author would also like to thank the Fakult\"at f\"ur Mathematik, Universit\"at Wien where his work was supported by the European Research Council (ERC) grant of Goulnara Arzhantseva, ``ANALYTIC" grant agreement no. 259527.

\section{Preliminaries, Definitions and Notations} 	
	The topic of rigidity of graphs has a long history. This can be attributed to its applications in a variety of areas like in mechanical frameworks, rigid structures, robotics etc.
	Laman studied it in the $1970$'s and gave a criterion for certain graphs to be movable.
	\begin{theorem}[Laman's criterion \cite{Lam70}]\label{Lamcrit}
	Let $\mathbf{\Gamma}$ be a finite graph with $|E|< 2|V|-3$. Then $\mathbf{\Gamma}$ is movable.
	\end{theorem}
	He infact, showed something stronger - the graph is movable if it does not contain any Laman subgraph with the same set of vertices. By a Laman subgraph one means $\mathbf{\Gamma} = (V,E)$ with $|E| = 2|V|-3$ and $|E'|\leqslant 2|V'|-3$ for any subgraph $\mathbf{\Gamma}' = (V',E')$ of $\mathbf{\Gamma}$. 
	This was originally a result of Pollaczek-Geiringer \cite{PG27} which was rediscovered by Laman. However, the above is a sufficient condition. There exist a lot of graphs which do not satisfy the above criteria but are movable. For example Dixon showed that \cite{Dix1899} bi-partite graphs are movable. Other works in related contexts can be found in  \cite{FJK15}, \cite{Sta14}, \cite{JJSS15}, \cite{MT01} etc. In conclusion, the study of rigidity and flexibility of non-bipartite graphs which do not satisfy Laman's criteria is interesting and shall be the setting of this article. We now collect some important definitions and notations which we shall need for the rest of the work.
	
	 Let $G$ be any group and $A$ be an arbitrary subset of $G$ (not necessarily symmetric and not necessarily containing the identity). The $h$-fold product set of $A$ is defined as $$A^{h} :=\lbrace a_{1}.a_{2}...a_{h} : a_{1},\ldots,a_h \in A \rbrace.$$ 
	Now, we recall the notion of a Cayley graph of a group.
	\begin{definition}[Cayley graph]
		Let $G$ be a finite group and $A$ be a symmetric generating set\footnote{a generating set in a group always means that any element of the group can be written as a product of finite elements from the set} of $G$. Then the Cayley graph $C(G,A)$ is the graph having the elements of $G$ as vertices and $\forall x,y\in G$ there is an (undirected) edge between $x$ and $y$ if and only if $\exists s\in A$ such that $sx=y$. 
	\end{definition}

	Although we shall not need it, we mention that the Cayley graphs can be defined for larger classes of groups, like all finitely generated groups. Further, the generating set $A$ need not be symmetric, in which case the graph is directed. Since we are dealing with undirected graphs, we choose $A= A^{-1}$. Also the identity $e\in G$ may be in $A$, in which case the graph has loops. We are dealing with loopless graphs, so for us $e\notin A$. 

	A graph is said to be $r$-regular (where $r\geqslant 1$ is an integer) if there are exactly $r$ half edges connected to each vertex.
	If $|A|= d$, it is clear that $C(G,A)$ will be $d$-regular (where $|A|$ denotes the cardinality of the set $A$). We shall sometimes refer to $r$-regular graphs as degree $r$ graphs.
	
	\begin{definition}[NAC-coloring and good NAC-coloring]\label{defNAC}
		Let $\mathbf{\Gamma}$ be a graph and $\mathcal{C}$ be a coloring of edges using two colors say red and blue. A cycle in $G$ is a red cycle if all its edges are red while it is an almost red cycle, if exactly one of its edges is blue, while all other edges are red. Blue cycles and almost blue cycles are defined similarly. Also the notion of red paths or of blue paths is now clear i.e., a path made up of only red edges or a path made up of only blue edges respectively.
		
		A coloring $\mathcal{C}$ is called a NAC-coloring, if it is surjective (i.e., uses both the colors) and there are no almost blue cycles or almost red cycles in $G$. It is a good NAC-coloring if it is a NAC-coloring and there does not exist a pair of distinct vertices which are joined by both a blue path and a red path.
	\end{definition}
	The definition of a NAC-coloring was introduced in \cite{GLS17}. One of the main results of their work was a combinatorial criteria for flexibility using NAC-colorings.
	\begin{theorem}[\cite{GLS17}]\label{ThmNACflex}
		A connected graph $G$ with at least one edge has a 
flexible labeling if and only if it
		has a NAC-coloring.
	\end{theorem} 

	In a subsequent work \cite{GLS19}, they showed a sufficient condition for movability:
\begin{proposition}[Lemma $3.2$ \cite{GLS19}]\label{keyLem1}
	Let $\mathcal{C}$ be a NAC-coloring of a graph $\mathbf{\Gamma}$. Let $R_{1}, \cdots, R_{m}$ be the sets of vertices of connected components of the graph obtained from $\mathbf{\Gamma}$ by keeping only red edges and  $B_{1}, \cdots,B_{n}$ be the same for the blue edges. If $|R_i \cap B_j | \leqslant 1 \,\forall 1\leqslant i \leqslant m, 1\leqslant j \leqslant n$ then $\mathbf{\Gamma}$ is movable.
\end{proposition}

Using the above Prop. \ref{keyLem1} one can conclude -

\begin{lemma}\label{keyLem2}
	A finite, simple graph is movable if there exists a good NAC coloring of the edges.
\end{lemma}
\begin{proof}
	The proof is direct once we recall the definition of a good NAC-coloring - Definition \ref{defNAC} and show that the existence of such a coloring is equivalent to the condition of Lemma \ref{keyLem1}. 
\end{proof}

	\section{Movability of Cayley graphs}

	 We will now state the conditions under which a Cayley graph becomes movable -  Theorem \ref{thmCay}.
		
	\begin{theorem}[Movability of Cayley graphs]\label{ThmCay}
			Let $G$ be a finite group with the identity element $e$ and $S:= \lbrace s_{1},s_{2},\cdots, s_{k}\rbrace\subseteq G$ be such that $\langle S\cup S^{-1} \rangle = G$, $e \notin S$. 
		\begin{enumerate}
			\item Let $k=1$. If $|G| \leqslant 3$, then $C(G,S\cup S^{-1})$ is rigid. If $|G|> 3$, then the cycle graph of $G = \langle s_{1},s_{1}^{-1} \rangle$ is movable.
			\item Let $k>1$. If $\exists s\in S$ such that, 
		$$\langle s,s^{-1}\rangle\cap \Big(S\setminus \lbrace s,s^{-1} \rbrace \Big)  = \emptyset \text{ and }\lbrace s,s^{-1} \rbrace \cap \Big\langle \Big( S\setminus \lbrace s,s^{-1}\rbrace\Big)\Big\rangle = \emptyset ,$$
		then $C(G, S\cup S^{-1} ) $ is flexible. Further, if 
		$$\langle s,s^{-1}\rangle\cap  \Big\langle \Big( S\setminus \lbrace s,s^{-1}\rbrace\Big)\Big\rangle = \lbrace e \rbrace,$$ then $C(G, S\cup S^{-1} ) $  is movable. 
		\item Let $k>1$. If there exists a partition $S_{1}\sqcup S_{2} $ of $S$ such that, 
		$$\langle S_{1}\rangle\cap \Big(S\setminus S_{1} \cup S_{1}^{-1} \Big)  = \emptyset \text{ and }(S_{1} \cup S_{1}^{-1} ) \cap \Big\langle \Big( S\setminus S_{1} \cup S_{1}^{-1} \Big)\Big\rangle = \emptyset ,$$
		then $G = \langle S\cup S^{-1} \rangle$ is flexible. Further, if 
		$$\langle S_{1}\rangle\cap  \Big\langle \Big( S\setminus S_{1} \cup S_{1}^{-1} \Big)\Big\rangle = \lbrace e \rbrace,$$ then $G$ is movable. 
		\end{enumerate}		
	\end{theorem}
	\begin{proof} Part $(1)$. Let $k=1$. We implicitly assume that $|G|>1$ since $|G|=1$ implies that $e$ is the only element in $G$ and according to our assumption $e\notin S$. 
			
			Let $S = \lbrace s_{1}\rbrace $ and $|G| \leqslant 3$. The only graphs to consider here are $K_{2}$ and $K_{3}$. It is easy to see that the complete graphs $K_{2}$ and $K_{3}$ are rigid.
			
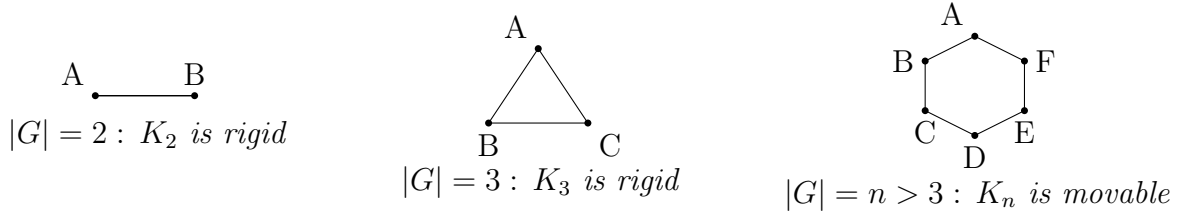
\begin{center}
	\begin{minipage}{.3\textwidth}
		
		\begin{tikzpicture}[scale=0.33] 
		
		\coordinate[label=above left:A] (A) at (-2,1);
	\coordinate[label=above:B] (B) at (2,1);
	\fill[black] (A) circle (4pt);
	\fill[black] (B) circle (4pt);
	\draw (A)--(B)--(A);
	\node at (0.5,-0.2,1) {\textit{$|G|=2: \, K_{2}$ is rigid }};		
		\end{tikzpicture} 		
	\end{minipage}\,\,	
	\begin{minipage}{.3\textwidth}		
		\begin{tikzpicture}[scale=0.33]
		
	\coordinate[label=above left:A] (A) at (0,3);
	\coordinate[label=below:B] (B) at (-2,0);
	\coordinate (C) at (2,0);
	\fill[black] (A) circle (4pt);
	\fill[black] (B) circle (4pt);
	\fill[black] (C) circle (4pt) node[below right]{C};
	\draw (A)--(B)--(C)--(A);
	\node at (0.5,-2,1) {\textit{$|G|=3: \, K_{3}$ is rigid}};
		
		\end{tikzpicture}
		
	\end{minipage}
	\begin{minipage}{.3\textwidth}
	
		\begin{tikzpicture}[scale=0.33] 
	\coordinate[label=above left:A] (A) at (0,4);
	\coordinate (B) at (-2,3);
	\coordinate (C) at (-2,1);
	\coordinate (D) at (0,0);
	\coordinate (E) at (2,1);
	\coordinate (F) at (2,3);
	\fill[black] (A) circle (4pt);
	\fill[black] (B) circle (4pt) node[left]{B};
	\fill[black] (C) circle (4pt) node[below]{C};
	\fill[black] (D) circle (4pt) node[below]{D};
	\fill[black] (E) circle (4pt) node[below]{E};
	\fill[black] (F) circle (4pt) node[right]{F};	
	\draw (A)--(B)--(C)--(D)--(E)--(F)--(A);
	\node at (0.5,-2,1) {\textit{$|G|=n>3: \, K_{n}$ is movable}};	
		\end{tikzpicture}	
	\end{minipage}
\end{center}

		For $|G| > 3$, the conclusion follows from Laman's criterion cf. Theorem \ref{Lamcrit} (we recall that, for the cyclic group generated by $S\cup S^{-1}$ with $S = \lbrace s_{1}\rbrace $, the number of edges of $C(G,S) = |G|$ which is less than $2|G|- 3, \forall G$, since $|G| > 3$).\\		
			\\				 
		 Part $(2)$.	Let $S = \lbrace s_{1},\cdots ,s_{k}\rbrace$ with $k>1$. Without loss of generality assume that $s_{1}= s$ and $s^{-1}\notin S$. Thus $\big(S\setminus\lbrace s,s^{-1}\rbrace\big) = \big(S\setminus\lbrace s_{1},s_{1}^{-1}\rbrace\big) = \lbrace s_{2},\cdots , s_{k}\rbrace.$\\
			\\
			First, let us check that the condition of movability, 	$$\langle s_{1}\rangle\cap  \langle  \lbrace s_{2},\cdots , s_{k}\rbrace\rangle = \lbrace e \rbrace,$$
			 automatically implies the condition on flexibility. Indeed, if $\langle s_{1}\rangle\cap  \Big\langle  \lbrace s_{2},\cdots , s_{k}\rbrace\Big\rangle = \lbrace e \rbrace$ then $\langle s_{1}\rangle\cap \lbrace s_{2},\cdots,s_{k}\rbrace = \emptyset$ (since $e\notin S$) and $\lbrace s_{1},s_{1}^{-1} \rbrace \cap \Big\langle \Big( S\setminus \lbrace s_{1},s_{1}^{-1}\rbrace\Big)\Big\rangle = \lbrace s_{1},s_{1}^{-1}\rbrace \cap
			 \langle \lbrace s_{2},\cdots , s_{k}\rbrace\rangle = \emptyset $ (since $s\neq e$).\\
			\\
			Next, we assume that $\langle s_{1}\rangle\cap \Big(S\setminus \lbrace s_{1},s_{1}^{-1} \rbrace \Big)  = \emptyset \text{ and }\lbrace s_{1},s_{1}^{-1} \rbrace \cap \Big\langle \Big( S\setminus \lbrace s_{1},s_{1}^{-1}\rbrace\Big)\Big\rangle = \emptyset.$\\			
		 Let us colour the edges of $C(G,S\cup S^{-1})$ with the colours red and blue in the following way - the edge corresponding to the element $s_{1}$ (and hence also $s_{1}^{-1}$) is coloured blue and the rest of the edges are coloured red. Then the above two conditions imply that there cannot exist a cycle which contains exactly one red edge or exactly one blue edge. Thus it is a NAC coloring and $C(G,S)$ is flexible by Theorem \ref{ThmNACflex}.
		 \begin{center}
		 	\begin{minipage}{.3\textwidth}
		 		\begin{tikzpicture}[scale=0.5] 
		 	\coordinate[label=above left:A] (A) at (0,8);
		 	\coordinate (B) at (-2,7.5);
		 	\coordinate (C) at (-3.5,6);
		 	\coordinate (D) at (-4,4);
		 	\coordinate (E) at (-3.5,2.5);
		 	\coordinate (F) at (-2.5,1);
		 	\coordinate (G) at (0,0);
		 	\coordinate (H) at (2,0.5);
		 	\coordinate (I) at (3.5,2);
		 	\coordinate (J) at (4,4);
		 	\coordinate (K) at (3.5,6);
		 	\coordinate (L) at (2,7.5);
		 	\fill[black] (A) circle (4pt);
		 	\fill[black] (B) circle (4pt) node[left]{B};
		 	\fill[black] (C) circle (4pt) node[left]{C};
		 	\fill[black] (D) circle (4pt) node[left]{D};
		 	\fill[black] (E) circle (4pt) node[below]{E};
		 	\fill[black] (F) circle (4pt) node[below]{F};
		 	\fill[black] (G) circle (4pt) node[below]{G};
		 	\fill[black] (H) circle (4pt) node[right]{H};
		 	\fill[black] (I) circle (4pt) node[right]{I};
		 	\fill[black] (J) circle (4pt) node[right]{J};
		 	\fill[black] (K) circle (4pt) node[right]{K};
		 	\fill[black] (L) circle (4pt) node[right]{L};
		 	\draw (A)--(C)--(E)--(G)--(I)--(K)--(A);
		 	\draw (B)--(D)--(F)--(H)--(J)--(L)--(B);
		 	\draw (A)--(D)--(G)--(J)--(A);
		 	\draw (B)--(E)--(H)--(K)--(B);
		 	\draw (C)--(F)--(I)--(L)--(C);
		 	\node at (0,-2,1) {\textit{$G= \mathbb{Z}/12\mathbb{Z}, S=\lbrace \bar{2},\bar{3}\rbrace$ is flexible}};	
		 	\end{tikzpicture}		
		 	\end{minipage}\,\,\,\,\,\,\,\,\,\,\,\,\,\,\,\,\,\,\,\,\,\,\,\,
		 	\begin{minipage}{.3\textwidth}
		 		
		 		\begin{tikzpicture}[scale=0.5] 
		 			\coordinate[label=above left:A] (A) at (0,8);
		 		\coordinate (B) at (-2,7.5);
		 		\coordinate (C) at (-3.5,6);
		 		\coordinate (D) at (-4,4);
		 		\coordinate (E) at (-3.5,2.5);
		 		\coordinate (F) at (-2.5,1);
		 		\coordinate (G) at (0,0);
		 		\coordinate (H) at (2,0.5);
		 		\coordinate (I) at (3.5,2);
		 		\coordinate (J) at (4,4);
		 		\coordinate (K) at (3.5,6);
		 		\coordinate (L) at (2,7.5);
		 		\fill[black] (A) circle (4pt);
		 		\fill[black] (B) circle (4pt) node[left]{B};
		 		\fill[black] (C) circle (4pt) node[left]{C};
		 		\fill[black] (D) circle (4pt) node[left]{D};
		 		\fill[black] (E) circle (4pt) node[below]{E};
		 		\fill[black] (F) circle (4pt) node[below]{F};
		 		\fill[black] (G) circle (4pt) node[below]{G};
		 		\fill[black] (H) circle (4pt) node[right]{H};
		 		\fill[black] (I) circle (4pt) node[right]{I};
		 		\fill[black] (J) circle (4pt) node[right]{J};
		 		\fill[black] (K) circle (4pt) node[right]{K};
		 		\fill[black] (L) circle (4pt) node[right]{L};
		 		\draw[color = blue] (A)--(C)--(E)--(G)--(I)--(K)--(A);
		 		\draw[color = blue] (B)--(D)--(F)--(H)--(J)--(L)--(B);
		 		\draw[color = red] (A)--(D)--(G)--(J)--(A);
		 		\draw[color = red] (B)--(E)--(H)--(K)--(B);
		 		\draw[color = red] (C)--(F)--(I)--(L)--(C);
		 		\node at (0.5,-2,1) {\textit{NAC coloring of $C(\mathbb{Z}/12\mathbb{Z},\lbrace \bar{\pm2},\bar{\pm3}\rbrace)$}};	
		 		\end{tikzpicture}	
		 	\end{minipage}	 
		 \end{center}
	 In the above figure, a flexible Cayley graph of the group $G=\mathbb{Z}/12\mathbb{Z}$ with respect to the generating set $S\cup S^{-1}=\lbrace \bar{\pm 2}, \bar{\pm 3}\rbrace$ is constructed. Note that the number of edges $|E|=24 > 2|V|-3$. This implies that it doesn't satisfy Laman's condition [i.e., Theorem \ref{Lamcrit}].\\
	 \\
		 Now, we assume that $\langle s\rangle\cap  \Big\langle \Big( S\setminus \lbrace s,s^{-1}\rbrace\Big)\Big\rangle = \lbrace e \rbrace$. We recall the sufficient condition for movability - Lemma \ref{keyLem2}. We color the edges of $G$ as follows: 
		 $$s_{1},s_{1}^{-1}\rightarrow  \,\,blue\,\,\, \& \,\,\, \big(S\setminus\lbrace s_{1},s_{1}^{-1}\rbrace\big) = \lbrace s_{2},\cdots , s_{k}\rbrace \rightarrow \,\, red.$$ 
		 Note that the Cayley graph is vertex transitive, so we only need to show that starting from the identity, there does not exist a vertex which can be connected by both a red path and a blue path. This is the same as saying that $\not\exists \,g\in G$ with $g\neq e$ such that $g=\lbrace s_{1},s_{1}^{-1}\rbrace^{m}$ and $g=\lbrace s_{2},\cdots,s_{k}\rbrace^{n}$ for all positive integers $m,n$. Our assumption on $S$ implies that this is true. Hence, the graph $C(G,S\cup S^{-1})$ is movable. \\
		 \\
		 The proof of part $(3)$ is similar to that of part $(2)$. Color the edges corresponding to $S_{1}\cup S_{1}^{-1}$ as blue and the edges of $S\setminus S_{1}\cup S_{1}^{-1}$ as red. Then the same argument as in the previous case shows that the graph is flexible under the assumption 
		 $$\langle S_{1}\rangle\cap \Big(S\setminus S_{1} \cup S_{1}^{-1} \Big)  = \emptyset \text{ and }(S_{1} \cup S_{1}^{-1} ) \cap \Big\langle \Big( S\setminus S_{1} \cup S_{1}^{-1} \Big)\Big\rangle = \emptyset.$$
		 While, if we assume that $$\langle S_{1}\rangle\cap  \Big\langle \Big( S\setminus S_{1} \cup S_{1}^{-1} \Big)\Big\rangle = \lbrace e \rbrace,$$ then $G$ is movable. 
		 			\end{proof}
		  \begin{center}
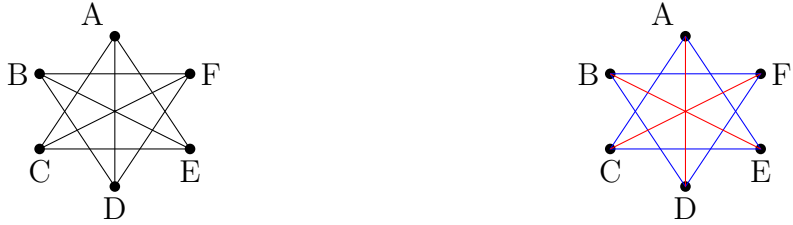
		 	\begin{minipage}{.3\textwidth}
		 		\begin{tikzpicture}[scale=0.5] 
		 	\coordinate[label=above left:A] (A) at (0,4);
		 	\coordinate (B) at (-2,3);
		 	\coordinate (C) at (-2,1);
		 	\coordinate (D) at (0,0);
		 	\coordinate (E) at (2,1);
		 	\coordinate (F) at (2,3);
		 	\fill[black] (A) circle (4pt);
		 	\fill[black] (B) circle (4pt) node[left]{B};
		 	\fill[black] (C) circle (4pt) node[below]{C};
		 	\fill[black] (D) circle (4pt) node[below]{D};
		 	\fill[black] (E) circle (4pt) node[below]{E};
		 	\fill[black] (F) circle (4pt) node[right]{F};	
		 	\draw (A)--(C)--(E)--(A);
		 	\draw (B)--(D)--(F)--(B);
		 	\draw (A)--(D) (B)--(E) (C)--(F);
		 	\node at (0.5,-2,1) {\textit{$G= \mathbb{Z}/6\mathbb{Z}, S=\lbrace 2,3\rbrace$ is movable}};	
		 	\end{tikzpicture}	
		 	\end{minipage}\,\,\,\,\,\,\,\,\,\,\,\,\,\,\,\,\,\,\,\,\,\,\,\,
		 	\begin{minipage}{.3\textwidth}
		 		
		 		\begin{tikzpicture}[scale=0.5] 
		 		\coordinate[label=above left:A] (A) at (0,4);
		 		\coordinate (B) at (-2,3);
		 		\coordinate (C) at (-2,1);
		 		\coordinate (D) at (0,0);
		 		\coordinate (E) at (2,1);
		 		\coordinate (F) at (2,3);
		 		\fill[black] (A) circle (4pt);
		 		\fill[black] (B) circle (4pt) node[left]{B};
		 		\fill[black] (C) circle (4pt) node[below]{C};
		 		\fill[black] (D) circle (4pt) node[below]{D};
		 		\fill[black] (E) circle (4pt) node[below]{E};
		 		\fill[black] (F) circle (4pt) node[right]{F};	
		 		\draw[color=blue] (A)--(C)--(E)--(A);
		 		\draw[color=blue] (B)--(D)--(F)--(B);
		 		\draw[color=red] (A)--(D) (B)--(E) (C)--(F);
		 		\node at (0.5,-2,1) {\textit{Good NAC coloring of $C(\mathbb{Z}/6\mathbb{Z},\lbrace \pm2,\pm3\rbrace)$}};	
		 		\end{tikzpicture}	
		 	\end{minipage}	 
		 \end{center}		 
		 
		 In the above figure, the given graph is movable. We easily check that $s_{1}=\bar{2},s_{1}^{-1}= \bar{-2}=\bar{4}, \langle s_{1} \rangle = \lbrace \bar{0},\bar{2},\bar{4}\rbrace$, while $s_{2}=\bar{3}, s_{2}^{-1}=\bar{-3}=\bar{3}, \langle s_{2}\rangle = \lbrace \bar{0},\bar{3}\rbrace.$ Hence, $\langle s_{1}\rangle \cap \langle s_{2}\rangle = \lbrace \bar{0}\rbrace.$

\begin{remark}
	The complete graph, with $S\cup S^{-1} = G$ cannot satisfy the conditions of Theorem \ref{thmCay}. In fact, it is rigid, not flexible.
\end{remark}
We turn to the corollary mentioned in the introduction.
\begin{proof}[Proof of Corollary \ref{CorA1}]
	Let $G$ be a group generated by the generating set $S=\lbrace s_{1},\cdots, s_{k}\rbrace$ with $\langle a_{i}\rangle \cap \langle a_{j}\rangle = \lbrace e \rbrace \,\, \forall 1\leqslant i\neq j\leqslant k.$ Coloring the edges $a_{1}, a_{1}^{-1}$ blue and the remaining edges red, we get the result. 
\end{proof}

\section{Explicit construction of dense, movable graphs}
In this section we construct movable Cayley graphs along with a good NAC coloring of the edges. First, we give a general criterion of movability for cartesian products of graphs 
\begin{proposition}[Movability under Cartesian products]\label{CartesianProdProp}
	There exists a good NAC coloring in the cartesian product of two finite, simple graphs. In particular, cartesian product of any two finite simple graphs is movable.
\end{proposition}

\begin{proof}
	Let $\mathbf{\Gamma}_{1}$ and $\mathbf{\Gamma}_{2}$ be any two finite, simple graphs. The cartesian product of $\mathbf{\Gamma}_{1}$ and $\mathbf{\Gamma}_{2}$, denoted by $\mathbf{\Gamma}=\mathbf{\Gamma}_{1}\boxdot \mathbf{\Gamma}_{2}$ is the graph having the vertex set 
	$$V_{\mathbf{\Gamma}_{1}\boxdot \mathbf{\Gamma}_{2}}:= \lbrace (x_{1},x_{2}) : x_{1}\in V_{\mathbf{\Gamma}_{1}}, x_{2}\in V_{\mathbf{\Gamma}_{2}}\rbrace$$
	and the edge set 
	$$E_{\mathbf{\Gamma}_{1}\boxdot \mathbf{\Gamma}_{2}} := \lbrace ((x_{1},x_{2}),(y_{1},y_{2})) :  x_{1}=y_{1}\, \& \,  (x_{2},y_{2})\in E_{\mathbf{\Gamma}_{2}} \text{ or } (x_{1},y_{1})\in E_{\mathbf{\Gamma}_{1}} \, \& \, x_{2}=y_{2} \rbrace$$
	where $(x_{1},x_{2}),(y_{1},y_{2})\in V_{\mathbf{\Gamma}_{1}\boxdot \mathbf{\Gamma}_{2}}.$ 
	We give a good NAC coloring of the edges of $\mathbf{\Gamma}_{1}\boxdot \mathbf{\Gamma}_{2}$. Color the edges of $\mathbf{\Gamma}_{1}$ by red and that of $\mathbf{\Gamma}_{2}$ by blue. This induces a coloring of $E_{\mathbf{\Gamma}_{1}\boxdot \mathbf{\Gamma}_{2}}$. First, we show that this is a NAC coloring. Let $C:= \lbrace (x_{11},x_{21}), (x_{12},x_{22}), \cdots, (x_{1k},x_{2k}), (x_{11},x_{21})\rbrace$ be a cycle in $\mathbf{\Gamma}$. The presence of exactly one red edge is equivalent to saying that there exists exactly one $i$ with $1\leqslant i\leqslant k$ such that $x_{1i} \neq x_{1(i+1)}$, while for all other $1\leqslant j\neq i\leqslant k$, we have $x_{1j}=x_{1(j+1)}$. This is clearly a contradiction to the fact that $C$ is a closed cycle. Thus there cannot exist any cycle with exactly one red edge (similarly blue edge). The coloring is an NAC coloring.  To show that it is a \textit{good} NAC coloring, assume that there exist two distinct vertices $(x_{1},x_{2})$ and $(y_{1},y_{2})$ of $\mathbf{\Gamma}$ such that they can be joined by both a red path and a blue path. According to our coloring, this is equivalent to saying that $x_{2} = y _{2}$ and $x_{1}=y_{1}$ respectively. This contradicts the fact that the vertices $(x_{1},x_{2})$ and $(y_{1},y_{2})$ are distinct. We have established a good NAC coloring of the edges of $\mathbf{\Gamma}$. The movability of $\mathbf{\Gamma}$ follows.
\end{proof}

\begin{theorem}[Movable graphs of all regularity in abelian groups]\label{MovRegAb}
	Fix a positive integer $r$ and let $1<q_{1}<\cdots< q_{r}$ be any $r$ pairwise relatively prime numbers. Let $\alpha_{1},\cdots,\alpha_{r}$ be positive integers and $n=q_{1}^{\alpha_{1}}q_{2}^{\alpha_{2}}\cdots q_{r}^{\alpha_{r}}$. Then the following hold:
	\begin{enumerate}
		\item The Cayley graph of 
		$$G=(\mathbb{Z}/q_{i}\mathbb{Z})^{\alpha_{i}}\,\forall 1\leqslant i \leqslant r, \,q_{i}>2$$ 
		with respect to the generating set $S\cup S^{-1}$ with 
		$$S=\lbrace (\bar{1},0,\cdots,0), (0,\bar{1},\cdots,0),\cdots, (0,0,\cdots, \bar{1})\rbrace$$ is a $2\alpha_{i}$
		 regular graph which is movable.
		\item The Cayley graph of 
		$$G=(\mathbb{Z}/n\mathbb{Z})\simeq (\mathbb{Z}/q_{1}^{\alpha_{1}}\mathbb{Z})\times \cdots \times (\mathbb{Z}/q_{r}^{\alpha_{r}}\mathbb{Z}) $$
		 with respect to the generating set $S\cup S^{-1}$ with 
		 $$S=\lbrace (\bar{1},0,\cdots,0), (0,\bar{1},\cdots,0),\cdots, (0,0,\cdots, \bar{1})\rbrace$$ is a $2r$ regular graph (for $q_{1}^{\alpha_{1}}>2$) which is movable.
		 \item The Cayley graph of 
		 $$G=(\mathbb{Z}/2\mathbb{Z})\times(\mathbb{Z}/q_{1}^{\alpha_{1}}\mathbb{Z})\times \cdots \times (\mathbb{Z}/q_{r}^{\alpha_{r}}\mathbb{Z}) $$
		 with respect to the generating set $S\cup S^{-1}$ with 
		 $$S=\lbrace (\bar{1},0,\cdots,0), (0,\bar{1},\cdots,0),\cdots, (0,0,\cdots, \bar{1})\rbrace$$ is a $2r+1$ regular graph which is movable.
	\end{enumerate}
	
	Conversely, given any finite abelian group $G$, one can find a symmetric generating set $A$ with respect to which the undirected Cayley graph $C(G,A)$ is a movable graph.
\end{theorem}
\begin{proof} Fix an integer $q>2$ and a positive integer $\alpha$. First note that $(\mathbb{Z}/q\mathbb{Z})^{\alpha} \not\simeq (\mathbb{Z}/q^{\alpha}\mathbb{Z})$. Hence the abelian group $(\mathbb{Z}/q\mathbb{Z})^{\alpha} $ which is a direct product of the cyclic groups $\mathbb{Z}/q\mathbb{Z}$ is non-cyclic and has (as a symmetric generating set) the set $A=S\cup S^{-1}$ where $S=\lbrace (\bar{1},0,\cdots,0), (0,\bar{1},\cdots,0),\cdots, (0,0,\cdots, \bar{1})\rbrace$. Since $q>2$, we have $|A|=|S\cup S^{-1}|=2|S|=2\alpha$. This implies that the Cayley graph $C(G,A)$ is $2\alpha$ regular. Now we construct a good NAC coloring of the edges via the map, 
	$$\lbrace (\bar{1},0,\cdots,0),(\bar{-1},0,\cdots,0)\rbrace \rightarrow \text{ blue and } (S\setminus\lbrace (\bar{1},0,\cdots,0),(\bar{-1},0,\cdots,0)\rbrace ) \rightarrow \text{ red}. $$
	It remains to show that $\langle \lbrace (\bar{1},0,\cdots,0),(\bar{-1},0,\cdots,0)\rbrace \rangle \cap \langle (S\setminus\lbrace (\bar{1},0,\cdots,0),(\bar{-1},0,\cdots,0)\rbrace )\rangle = \lbrace (0,\cdots,0)\rbrace$. But this is clear from the fact that we have a direct product of the groups (in fact, one can exploit this strategy to construct movable graphs having more edges see Theorem \ref{DenseMovCay}). In the first figure below, the Cayley graph for $q=3$ and $\alpha = 2$ is shown.
		
			\begin{minipage}{.4\textwidth}
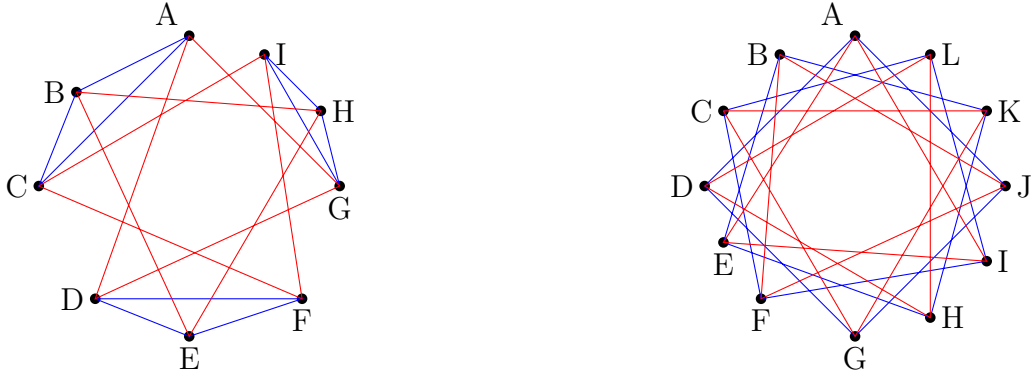
				\begin{tikzpicture}[scale=0.5] 
				\coordinate[label=above left:A] (A) at (0,8);
				\coordinate (B) at (-3,6.5);
				\coordinate (C) at (-4,4);
				\coordinate (D) at (-2.5,1);
				\coordinate (E) at (0,0);
				\coordinate (F) at (3,1);
				\coordinate (G) at (4,4);
				\coordinate (H) at (3.5,6);
				\coordinate (I) at (2,7.5);
				\fill[black] (A) circle (4pt);
				\fill[black] (B) circle (4pt) node[left]{B};
				\fill[black] (C) circle (4pt) node[left]{C};
				\fill[black] (D) circle (4pt) node[left]{D};
				\fill[black] (E) circle (4pt) node[below]{E};
				\fill[black] (F) circle (4pt) node[below]{F};
				\fill[black] (G) circle (4pt) node[below]{G};
				\fill[black] (H) circle (4pt) node[right]{H};
				\fill[black] (I) circle (4pt) node[right]{I};
				\draw[color=blue] (A)--(B)--(C)--(A) (D)--(E)--(F)--(D) (G)--(H)--(I)--(G) ;
				\draw[color = red] (A)--(D)--(G)--(A) (B)--(E)--(H)--(B) (C)--(F)--(I)--(C);

				\node at (0,-2,1) {\textit{$G= (\mathbb{Z}/3\mathbb{Z})^{2}, S=\lbrace (\bar{1},\bar{0}),(\bar{0},\bar{1})\rbrace$ is movable}};	
				\end{tikzpicture}		
			\end{minipage}\,\,\,\,\,\,\,\,\,\,\,\,\,\,\,\,\,\,\,\,\,\,\,
			\begin{minipage}{.5\textwidth}
				
				\begin{tikzpicture}[scale=0.5] 
				\coordinate[label=above left:A] (A) at (0,8);
				\coordinate (B) at (-2,7.5);
				\coordinate (C) at (-3.5,6);
				\coordinate (D) at (-4,4);
				\coordinate (E) at (-3.5,2.5);
				\coordinate (F) at (-2.5,1);
				\coordinate (G) at (0,0);
				\coordinate (H) at (2,0.5);
				\coordinate (I) at (3.5,2);
				\coordinate (J) at (4,4);
				\coordinate (K) at (3.5,6);
				\coordinate (L) at (2,7.5);
				\fill[black] (A) circle (4pt);
				\fill[black] (B) circle (4pt) node[left]{B};
				\fill[black] (C) circle (4pt) node[left]{C};
				\fill[black] (D) circle (4pt) node[left]{D};
				\fill[black] (E) circle (4pt) node[below]{E};
				\fill[black] (F) circle (4pt) node[below]{F};
				\fill[black] (G) circle (4pt) node[below]{G};
				\fill[black] (H) circle (4pt) node[right]{H};
				\fill[black] (I) circle (4pt) node[right]{I};
				\fill[black] (J) circle (4pt) node[right]{J};
				\fill[black] (K) circle (4pt) node[right]{K};
				\fill[black] (L) circle (4pt) node[right]{L};
				\draw[color = blue] (A)--(J)--(G)--(D)--(A) (B)--(K)--(H)--(E)--(B) (C)--(F)--(I)--(L)--(C) ;

				\draw[color = red] (A)--(E)--(I)--(A) (B)--(F)--(J)--(B) (C)--(G)--(K)--(C) (D)--(H)--(L)--(D);
				\node at (0.5,-2,1) {\textit{$G= (\mathbb{Z}/4\mathbb{Z})\times{(\mathbb{Z}/3\mathbb{Z})}, S=\lbrace (\bar{1},\bar{0}),(\bar{0},\bar{1})\rbrace$ is movable}};	
				\end{tikzpicture}	
			\end{minipage}\\
		\\	 

Now, note that $(\mathbb{Z}/n\mathbb{Z})\simeq (\mathbb{Z}/q_{1}^{\alpha_{1}}\mathbb{Z})\times \cdots \times (\mathbb{Z}/q_{r}^{\alpha_{r}}\mathbb{Z})$ by the Chinese remainder theorem. We adopt the same method as in the previous case and take as generating set $A=S\cup S^{-1}$ where $S=\lbrace (\bar{1},0,\cdots,0), (0,\bar{1},\cdots,0),\cdots, (0,0,\cdots, \bar{1})\rbrace$. Here the graphs will be $2r$ regular while a good NAC coloring of the edges is given by $$\lbrace (\bar{1},0,\cdots,0),(\bar{-1},0,\cdots,0)\rbrace \rightarrow \text{ blue and } (S\setminus\lbrace (\bar{1},0,\cdots,0),(\bar{-1},0,\cdots,0)\rbrace ) \rightarrow \text{ red}. $$
In the second figure above, we have $G=\mathbb{Z}/12\mathbb{Z}\simeq (\mathbb{Z}/4\mathbb{Z})\times{(\mathbb{Z}/3\mathbb{Z})}$ with the symmetric generating set $A=\lbrace (\bar{1},\bar{0}),(\bar{3},\bar{0}),(\bar{0},\bar{2}),(\bar{0},\bar{1})\rbrace$. Note that the underlying group is the same as in example diagram of Theorem \ref{ThmCay} but the generating sets are different. Here the generating set (when translated in terms of $\mathbb{Z}/12\mathbb{Z}$) is $\lbrace \bar{9},\bar{3},\bar{4},\bar{8}\rbrace$.

Uptil now, we have constructed movable abelian Cayley graphs of degree $2r$ for any positive integer $r$. To get the graphs of odd regularity we need to add involutions to the generating set (we recall that an involution is an element of order $2$ in a group). We consider $G=(\mathbb{Z}/2\mathbb{Z})\times(\mathbb{Z}/q_{1}^{\alpha_{1}}\mathbb{Z})\times \cdots \times (\mathbb{Z}/q_{r}^{\alpha_{r}}\mathbb{Z}) $
with respect to the generating set $S\cup S^{-1}$ with 
$$S=\lbrace (\bar{1},0,\cdots,0), (0,\bar{1},\cdots,0),\cdots, (0,0,\cdots, \bar{1})\rbrace.$$ 
Arguing as before we have a movable, $2r+1$ regular, graph.

Cases $(1),(2),(3)$ cover the construction of movable, regular graphs of all degrees as claimed.\\
\\
For the converse part, let $G$ be any finite abelian group. By the fundamental theorem of finitely generated abelian groups, $$G\simeq (\mathbb{Z}/p_{1}^{\alpha_{1}}\mathbb{Z})\times \cdots \times (\mathbb{Z}/p_{t}^{\alpha_{t}}\mathbb{Z}) $$
for primes $p_{1}<\cdots <p_{t}$ and positive integers $\alpha_{1},\cdots,\alpha_{t}$. Choose $A=S\cup S^{-1}$ with \\
$S=\lbrace (\bar{1},0,\cdots,0), (0,\bar{1},\cdots,0),\cdots, (0,0,\cdots, \bar{1})\rbrace.$ This makes $C(G,A)$ a movable Cayley graph.
\end{proof}

We are now in a position to construct - 

\begin{theorem}[Generic construction of dense movable graphs]\label{DenseMovCay}
	For each positive integer $n\geqslant 2$, one can construct $2n-2$ regular, movable graphs with $n^{2}$ vertices and $n^{3}-n^{2}$ edges. Also, there exist an absolute constant $c>0$ such that there is a sequence of abelian Cayley graphs having $N$ (with $N\rightarrow \infty$) vertices and more than $cN^{2}$ edges.
\end{theorem}

\begin{proof}
	Before we proceed with the proof, we mention that there is a distinction between the two statements of the theorem. A priori, the second statement already ensures the construction of a denser sequence of movable graphs but in the first statement one can get the stronger control on the number of vertices of each graph in the sequence but at the cost of reducing the density $\alpha$. Now, we proceed with the proof.\\ 
	We construct the graphs as Cayley graph of the abelian group $G=(\mathbb{Z}/n\mathbb{Z})\times (\mathbb{Z}/n\mathbb{Z})$. Consider the symmetric generating set 
	$$A=[((\mathbb{Z}/n\mathbb{Z})\times \lbrace 0\rbrace) \setminus \lbrace 0,0\rbrace] \cup [(\lbrace 0\rbrace \times (\mathbb{Z}/n\mathbb{Z}))\setminus \lbrace 0,0\rbrace]$$ of $G$. It is clear that $A$ is indeed a generating set and also the fact that $|A|= 2n-2$. Also $|G|=n^{2}$. Hence, the number of edges of $G$ is $n^{2}(n-1)$. A good NAC coloring of the edges is given by 
	\begin{align*}
		[((\mathbb{Z}/n\mathbb{Z})\times \lbrace 0\rbrace) \setminus \lbrace 0,0\rbrace] & \rightarrow \text{ blue}\\
		[(\lbrace 0\rbrace \times (\mathbb{Z}/n\mathbb{Z}))\setminus \lbrace 0,0\rbrace] & \rightarrow \text{ red}
	\end{align*}
We easily check that $\langle 	[((\mathbb{Z}/n\mathbb{Z})\times \lbrace 0\rbrace) \setminus \lbrace 0,0\rbrace] \rangle \cap \langle [(\lbrace 0\rbrace \times (\mathbb{Z}/n\mathbb{Z}))\setminus \lbrace 0,0\rbrace] \rangle = \lbrace 0,0\rbrace $. Thus the graph is movable. We illustrate with an example for $n=4$ below.

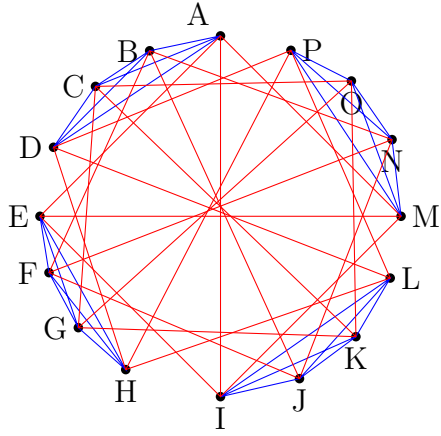
\begin{center}
	\begin{tikzpicture}[scale=0.6] 
	\coordinate[label=above left:A] (A) at (0,8);
	\coordinate (B) at (-1.57,7.67);
	\coordinate (C) at (-2.77,6.88);
	\coordinate (D) at (-3.7,5.53);
	\coordinate (E) at (-4,4);
	\coordinate (F) at (-3.8,2.75);
	\coordinate (G) at (-3.15,1.53);
	\coordinate (H) at (-2.1,0.6);
	\coordinate (I) at (0,0);
	\coordinate (J) at (1.75,0.4);
	\coordinate (K) at (3,1.33);
	\coordinate (L) at (3.76, 2.63);
	\coordinate (M) at (4,4);
	\coordinate (N) at (3.8,5.7);
	\coordinate (O) at (2.9,7);
	\coordinate (P) at (1.56, 7.68);
	\fill[black] (A) circle (3pt);
	\fill[black] (B) circle (3pt) node[left]{B};
	\fill[black] (C) circle (3pt) node[left]{C};
	\fill[black] (D) circle (3pt) node[left]{D};
	\fill[black] (E) circle (3pt) node[left]{E};
	\fill[black] (F) circle (3pt) node[left]{F};
	\fill[black] (G) circle (3pt) node[left]{G};
	\fill[black] (H) circle (3pt) node[below]{H};
	\fill[black] (I) circle (3pt) node[below]{I};
	\fill[black] (J) circle (3pt) node[below]{J};
	\fill[black] (K) circle (3pt) node[below]{K};
	\fill[black] (L) circle (3pt) node[right]{L};
	\fill[black] (M) circle (3pt) node[right]{M};
	\fill[black] (N) circle (3pt) node[below]{N};
	\fill[black] (O) circle (3pt) node[below]{O};
	\fill[black] (P) circle (3pt) node[right]{P};
	\draw[color=blue] (A)--(B)--(C)--(D)--(A)	(B)--(D) (C)--(A);
	\draw[color=blue] (E)--(F)--(G)--(H)--(E)	(F)--(H) (G)--(E);
	\draw[color=blue] (I)--(J)--(K)--(L)--(I)	(J)--(L) (K)--(I);
	\draw[color=blue] (M)--(N)--(O)--(P)--(M)	(N)--(P) (O)--(M);
	\draw[color=red] (A)--(E)--(I)--(M)--(A)	(E)--(M) (I)--(A);
	\draw[color=red] (B)--(F)--(J)--(N)--(B)	(B)--(J) (F)--(N);
	\draw[color=red] (C)--(G)--(K)--(O)--(C)	(C)--(K) (G)--(O);
	\draw[color=red] (D)--(H)--(L)--(P)--(D)	(D)--(L) (P)--(H);
	\node at (0,-2,1) {\textit{$G= (\mathbb{Z}/4\mathbb{Z})^{2}, A=\lbrace (\mathbb{Z}/4\mathbb{Z},\bar{0}),(\bar{0},\mathbb{Z}/4\mathbb{Z})\rbrace\setminus\lbrace \bar{0},\bar{0}\rbrace.$}};	
	\end{tikzpicture}		
\end{center}

We now address the further part. Fix $n\geqslant 2$ and consider the Cayley graphs of $G=(\mathbb{Z}/n\mathbb{Z})\times (\mathbb{Z}/n^{k}\mathbb{Z})$ where $k\in \mathbb{N}$ with respect to the generating set
$$A=[((\mathbb{Z}/n\mathbb{Z})\times \lbrace 0\rbrace) \setminus \lbrace 0,0\rbrace] \cup [(\lbrace 0\rbrace \times (\mathbb{Z}/n^{k}\mathbb{Z}))\setminus \lbrace 0,0\rbrace]$$ of $G$. It is clear that $A$ is indeed a generating set and also the fact that $|A|= n + n^{k}-2$. Also $|G|=n^{k+1}$. Hence, the number of edges of $C(G,A)$ is $\frac{n^{k+1}(n+n^{k}-2)}{2} = |G|^{2}(\frac{1}{2n^{k}}+\frac{1}{2n}-\frac{1}{n^{k+1}})$. A good NAC coloring of the edges is given by 
\begin{align*}
[((\mathbb{Z}/n\mathbb{Z})\times \lbrace 0\rbrace) \setminus \lbrace 0,0\rbrace] & \rightarrow \text{ blue}\\
[(\lbrace 0\rbrace \times (\mathbb{Z}/n^{k}\mathbb{Z}))\setminus \lbrace 0,0\rbrace] & \rightarrow \text{ red}
\end{align*}
We easily check that $\langle 	[((\mathbb{Z}/n\mathbb{Z})\times \lbrace 0\rbrace) \setminus \lbrace 0,0\rbrace] \rangle \cap \langle [(\lbrace 0\rbrace \times (\mathbb{Z}/n^{k}\mathbb{Z}))\setminus \lbrace 0,0\rbrace] \rangle = \lbrace 0,0\rbrace $. Thus the graph is movable. 

We have constructed movable Cayley graphs with $N= n^{k+1}$ vertices and $N^{2}(\frac{1}{2n^{k}}+\frac{1}{2n}-\frac{1}{n^{k+1}})$ edges. Let $\frac{1}{2n}> c > 0$ and $ k \rightarrow \infty$. Then the number of edges of $C(G,A)$ is  $\geqslant cN^{2}$. Thus, as $k\rightarrow \infty$, we have a sequence of finite, moving, regular, simple graphs on $N$ vertices and more than $cN^{2}$ edges. This give us asymptotically the densest possible regular, moving graphs.
\end{proof}

\begin{remark}
	The construction in the previous theorem can be made more general using $G=(\mathbb{Z}/m\mathbb{Z})\times (\mathbb{Z}/n^{k}\mathbb{Z})$. In that case, the number of vertices becomes $mn^{k}$ and the number of edges $\frac{mn^{k}(m+n^{k}-2)}{2}$. This is actually the same as taking the cartesian products of the complete Cayley graphs of $(\mathbb{Z}/m\mathbb{Z})$ and $(\mathbb{Z}/n^{k}\mathbb{Z})$ [see Proposition \ref{CartesianProdProp}].
\end{remark}

We move onto construction of movable graphs in non-abelian groups and in particular the finite simple groups of Lie type.
\begin{theorem}[Movable graphs in non-abelian groups]\label{DensestMovCay}
	Let $G= SL_{n}(\mathbb{F}_{p})$ where $p$ is a prime and  $\mathbb{F}_{p}=\mathbb{Z}/p\mathbb{Z}$ denotes a finite field with $p$ elements. Then the following hold:	
	\begin{enumerate}
		\item Let $\forall 1\leqslant i, j\leqslant n $  
		$$e_{i,j}:= \text{ the matrix with 1 in the (i,j)th position and zero elsewhere }$$ while $E_{i,j} := I + e_{i,j}$ denote the elementary matrices. Then the Cayley graphs $C(G,A)$ with $A=S\cup S^{-1}$ where $S = \lbrace E_{i,i+1}, E_{i+1,i}: 1\leqslant i \leqslant n\rbrace $ are movable of regularity $4(n-1)$ (for $p>2$) and of regularity $2(n-1)$ for $p=2$.
		\item  One can construct dense movable graphs on $|SL_{n}(\mathbb{F}_{p})|=\frac{1}{p-1}\prod_{k=0}^{k=n-1}(p^{n}-p^{k})$ vertices of regularity $2(p^{\frac{n(n-1)}{2}}-1)$ and number of edges $N(p^{\frac{n(n-1)}{2}}-1)$. In particular, keeping $p$ fixed and $n\rightarrow \infty$, we get regular, movable graphs on $N$ vertices and with $O(N^{\frac{3}{2}})$ edges. 
		\item There exists a constant $c$ such that  one can construct a sequence of movable Cayley graphs coming from finite, simple special linear groups having $N$ vertices and more than $cN^{2}$ edges. Thus, asymptotically we obtain regular, movable graphs on $N$ vertices and $O(N^{2})$ edges in non-abelian groups.
	\end{enumerate}
\end{theorem}

\begin{proof}
	First part:	We first show that the Cayley graph $C(G,A)$ is connected, i.e., the set $A= S\cup S^{-1}$ where $S = \lbrace E_{i,i+1}, E_{i+1,i}: 1\leqslant i \leqslant n\rbrace $ is a generating set for $G= SL_{n}(\mathbb{F}_{p})$. Notice  that $(E_{i,j})^{\gamma}= I + \gamma e_{i,j},\,\forall \gamma \in \mathbb{Z}\setminus \lbrace 0\rbrace$. So if we can construct all the elementary matrices $E_{ij}$ from the set $A$ then we are done. The commutator $[E_{1,2},E_{2,3}] = E_{1,3}$ and in fact, for all $1\leqslant i \leqslant n$, 
		\begin{align}
		E_{i,i+2} &= [E_{i,i+1},E_{i+1,i+2}] \notag\\
		E_{i,i+3} &= [E_{i,i+2},E_{i+2,i+3}]\notag\\
		&\;\;\vdots \notag \\
		E_{i,n} &= [E_{i,n-1},E_{n-1,n}]\notag
		\end{align}
		Thus we get all the $E_{i,j}$'s in the upper triangular portion. A similar argument (considering the matrices $E_{i+1,i}$) gives us all the $E_{i,j}$'s in the lower triangular portion as well. This proves that $A=S\cup S^{-1}$ is a symmetric generating set of $G$. Now we establish a good NAC coloring of the elements in $A$. \\
		Let $S_{1}= \lbrace E_{i,i+1}:1\leqslant i\leqslant n \rbrace \subset S$ and  $S_{2}= \lbrace E_{i+1,i}:1\leqslant i\leqslant n \rbrace \subset S$. Let $A_{1}=S_{1}\cup S_{1}^{-1}$ and $A_{2}=S_{2}\cup S_{2}^{-1}$. It is easy to see that $A=A_{1}\cup A_{2}.$ Color the edges by
		$$A_{1}\rightarrow \text{ blue and } A_{2}\rightarrow \text{ red}$$
		We have that $\langle A_{1}\rangle = $ subgroup of all upper triangular matrices in $SL_{n}(\mathbb{F}_{p})$ while $\langle A_{2}\rangle = $ subgroup of all lower triangular matrices in $SL_{n}(\mathbb{F}_{p})$. Thus $\langle A_{1}\rangle \cap \langle A_{2}\rangle = \lbrace I \rbrace $. Hence, we have that the coloring is a good NAC coloring and the graph is movable. To compute the regularity we need to compute $|S\cup S^{-1}|$. If $p=2$, $S=S^{-1}$. Hence the regularity is $2(n-1)$. If $p>2$, then $S\cap S^{-1}= \emptyset$. This implies that $|S\cup S^{-1}| = 2|S| = 4(n-1)$.\\
		\\
		Proof of second part: From the above, we also see that we can extend the generating set $A$. 
		Consider the generating set $A'=S'\cup S'^{-1}$ where $S'=S_{1}'\cup S_{2}'$ and
		\begin{align*}
		S_{1}' & :=\lbrace A\in SL_{n}(\mathbb{F}_{p}): A_{i,i}=1, A_{i,j}=0\,\, \forall \, 1\leqslant i<j\leqslant n\rbrace \setminus \lbrace I \rbrace\\
		S_{2}' & :=  \lbrace A\in SL_{n}(\mathbb{F}_{p}): A_{i,i}=1, A_{i,j}=0\,\, \forall \, 1\leqslant j<i\leqslant n\rbrace \setminus \lbrace I \rbrace 
		\end{align*}
		Let as before $A_{1}' = S_{1}'\cup S_{1}'^{-1}$ and $A_{2}' = S_{2}'\cup S_{2}'^{-1}$. Then a good NAC coloring of the generating set $A'= A_{1}'\cup A_{2}'$ is given by
		$$A_{1}'\rightarrow \text{ blue and } A_{2}'\rightarrow \text{ red}.$$
		To compute the bounds on the number of vertices and the number of edges, note that the size of the group $|SL_{n}(\mathbb{F}_{p})| = \frac{1}{p-1}\prod_{k=0}^{k=n-1}(p^{n}-p^{k})$. Hence, the number of vertices is $N=\frac{1}{p-1}\prod_{k=0}^{k=n-1}(p^{n}-p^{k}).$ The graphs are regular and the regularity is given by $|A'|=|A_{1}'\cup A_{2}'|=|S_{1}'|+|S_{2}'|= 2(p^{\frac{n(n-1)}{2}}-1).$ We know that the number of edges in an $m$ vertex, $d$ regular graph is $\frac{md}{2}$. From this we see that the number of edges in $C(G,A')$ is $N(p^{\frac{n(n-1)}{2}}-1)$. Keeping $p$ fixed and $n\rightarrow \infty$ we see that $p^{\frac{n(n-1)}{2}}-1 = O(N^{\frac{1}{2}})$. Thus the number of edges is $O(N^{\frac{3}{2}})$ as $n\rightarrow \infty$ in the $N$ vertex Cayley graph.\\
		\\
		Proof of third part: Fix a prime $p$, natural numbers $n\geqslant 2, k\geqslant 1$. We shall use Proposition \ref{CartesianProdProp} to construct cartesian products of the complete Cayley graphs of $SL_{n}(\mathbb{F}_{p})$ and $SL_{n^{k}}(\mathbb{F}_{p})$. Let $G=SL_{n}(\mathbb{F}_{p})\times SL_{n^{k}}(\mathbb{F}_{p})$ and the symmetric generating set 
		$$A:=[(SL_{n}(\mathbb{F}_{p})\times \lbrace I_{n^{k}}\rbrace) \setminus \lbrace I_{n},I_{n^{k}}\rbrace] \cup [(\lbrace I_{n}\rbrace \times SL_{n^{k}}(\mathbb{F}_{p}))\setminus \lbrace I_{n},I_{n^{k}}\rbrace].$$		
		We now color the edges of $C(G,A)$ by 	
		\begin{align*}
			(SL_{n}(\mathbb{F}_{p})\times \lbrace I_{n^{k}}\rbrace) \setminus \lbrace I_{n},I_{n^{k}}\rbrace \rightarrow & \text{ blue, }\\
			(\lbrace I_{n}\rbrace \times SL_{n^{k}}(\mathbb{F}_{p}))\setminus \lbrace I_{n},I_{n^{k}}\rbrace\rightarrow & \text{ red}.
		\end{align*} 
		This is a good NAC coloring and $C(G,A)$ is movable. Now for the bounds. To avoid messy computations using constants we shall use an asymptotic order argument. Let $M_{1}=|SL_{n}(\mathbb{F}_{p})| = O(p^{n^{2}}), M_{2}= |SL_{n^{k}}(\mathbb{F}_{p})| = O(p^{n^{2k}})$. We know that $|G| = M_{1}M_{2} = O(p^{n^{2k}+n^{2}})$ and the degree of $C(G,A) = M_{1}+M_{2} -2 = O(p^{n^{2k}}).$ Thus the number of edges of $C(G,A)$, 
		$$|E| = \frac{1}{2}M_{1}M_{2}(M_{1}+M_{2}-2) = O(p^{2n^{2k}+n^{2}}) = O(p^{n^{2}})O(p^{2n^{2k}}).$$
		Keeping $n,p$ fixed and $k\rightarrow \infty$, we get a sequence of regular graphs on $N = |G|$ vertices and $O(N^{2})$ edges [since $|E| = O(p^{2n^{2k}+n^{2}}) = O(p^{n^{2k}+n^{2}}).O(p^{n^{2k}+n^{2}}) = O(N).O(N)=O(N^{2})$, when $p,n$ are fixed]. We have shown that there exists a constant $c$ such that  one can construct a sequence of movable Cayley graphs coming from finite, simple special linear groups having $N$ vertices and more than $cN^{2}$ edges.		

 As an illustration, we draw an explicit good NAC coloring of the Cayley graph of the group $SL_{2}(\mathbb{F}_{3})$ with respect to the generating set $A = S\cup S^{-1}$ where $S = \Big\lbrace \begin{pmatrix}
 1 & 1\\
 0 & 1
 \end{pmatrix}, \begin{pmatrix}
 1 & 0 \\
 1 & 1
 \end{pmatrix}\Big\rbrace .$\\

\begin{center}

		\begin{tikzpicture}[scale=0.8] 
		\coordinate[label=above left:A] (A) at (0,12);
		\coordinate (B) at (-1.6,11.8);
		\coordinate (C) at (-2.82,11.3);
		\coordinate (D) at (-4.2,10.3);
		\coordinate (E) at (-5.1,9.15);
		\coordinate (F) at (-5.7,7.8);
		\coordinate (G) at (-6.5,6);
		\coordinate (H) at (-6.8,4.5);
		\coordinate (I) at (-6.2,3.1);
		\coordinate (J) at (-4.33,1.8);
		\coordinate (K) at (-3.17,1);
		\coordinate (L) at (-1.92, 0.3);
		\coordinate (M) at (0,0);
		\coordinate (N) at (1.55,0.15);
		\coordinate (O) at (4,0.9);
		\coordinate (P) at (5.16,1.7);
		\coordinate (Q) at (5.1,2.9);
		\coordinate (R) at (5.8, 4.5);
		\coordinate (S) at (6,6);
		\coordinate (T) at (5.8,7.6);
		\coordinate (U) at (5.4,9);
		\coordinate (V) at (4.6,10.1);
		\coordinate (W) at (3.6,11);
		\coordinate (X) at (1.85,11.7);
		\fill[black] (A) circle (4pt);
		\fill[black] (B) circle (4pt) node[left]{B};
		\fill[black] (C) circle (4pt) node[left]{C};
		\fill[black] (D) circle (4pt) node[left]{D};
		\fill[black] (E) circle (4pt) node[left]{E};
		\fill[black] (F) circle (4pt) node[left]{F};
		\fill[black] (G) circle (4pt) node[left]{G};
		\fill[black] (H) circle (4pt) node[left]{H};
		\fill[black] (I) circle (4pt) node[left]{I};
		\fill[black] (J) circle (4pt) node[left]{J};
		\fill[black] (K) circle (4pt) node[left]{K};
		\fill[black] (L) circle (4pt) node[left]{L};
		\fill[black] (M) circle (4pt) node[left]{M};
		\fill[black] (N) circle (4pt) node[below]{N};
		\fill[black] (O) circle (4pt) node[below]{O};
		\fill[black] (P) circle (4pt) node[below]{P};
		\fill[black] (Q) circle (4pt) node[right]{Q};
		\fill[black] (R) circle (4pt) node[right]{R};
		\fill[black] (S) circle (4pt) node[right]{S};
		\fill[black] (T) circle (4pt) node[right]{T};
		\fill[black] (U) circle (4pt) node[right]{U};
		\fill[black] (V) circle (4pt) node[right]{V};
		\fill[black] (W) circle (4pt) node[right]{W};
		\fill[black] (X) circle (4pt) node[right]{X};
	\draw[color=blue] (A)--(B)--(C)--(A) (D)--(E)--(F)--(D)	(V)--(L)--(G)--(V) (P)--(J)--(H)--(P) (N)--(X)--(I)--(N) (U)--(R)--(K)--(U) (W)--(Q)--(M)--(W) (T)--(O)--(S)--(T) ;
\draw[color=red] (A)--(X)--(W)--(A)	(T)--(B)--(G)--(T) (H)--(C)--(R)--(H) (U)--(D)--(V)--(U) (I)--(E)--(S)--(I) (F)--(Q)--(J)--(F) (K)--(M)--(O)--(K) (L)--(N)--(P)--(L);

		\node at (0,-2,1) {\textit{$G= SL_{2}(\mathbb{F}_{3}), A = S\cup S^{-1}$ where $S = \Big\lbrace \begin{pmatrix}
				1 & 1\\
				0 & 1
				\end{pmatrix}, \begin{pmatrix}
				1 & 0 \\
				1 & 1
				\end{pmatrix}\Big\rbrace$. $C(G,A)$ is movable.}};	
		\end{tikzpicture}		
\end{center}
We note that in the above example, the number of edges $|E| = 48> 2|V|-3 $, so it doesn't satisfy Laman's criterion cf. Theorem \ref{Lamcrit}. \end{proof}

Finally, we give the proof of the existence of graphs of all regularity coming from non-abelian special linear groups.
\begin{proof}[Proof of corollary \ref{CorA2}]
	Let $r$ denote the regularity of the graph. Let $\forall 1\leqslant i, j\leqslant n $ and $1\leqslant k\neq i, l\neq j\leqslant n$, 
	$$e_{i,j}:= \lbrace A\in SL_{n}(\mathbb{F}_{p}): A_{i,j}=1, A_{k,l}=0 \rbrace $$ while $E_{i,j} := I + e_{i,j}$.
	\begin{enumerate}
		\item $r=2\rho$ (even) with $\rho \geqslant 1$. Consider $G=SL_{\rho+1}(\mathbb{F}_{2})$ and consider the symmetric generating set  $S\cup S^{-1}$ with  $S = \lbrace E_{i,i+1}, E_{i+1,i}: 1\leqslant i \leqslant n\rbrace $. Clearly $S=S^{-1}$, hence, $|S\cup S^{-1}| = |S| = 2\rho = r$. This is a moving graph by Theorem \ref{DensestMovCay} and the graph is of regularity $r$.
		\item $r=2\rho + 1$ (odd) with $\rho \geqslant 1$.  Consider $G=SL_{\rho+1}(\mathbb{F}_{2})$ with the symmetric generating set  $S\cup S^{-1}$ with  $S = \lbrace E_{i,i+1}, E_{i+1,i}: 1\leqslant i \leqslant n\rbrace \cup E_{1,3} $. Then $S=S^{-1}$ and $|S\cup S^{-1}| = |S| = 2\rho +1 = r.$ This is also a moving graph by Theorem \ref{DensestMovCay} and we are done.
	\end{enumerate}\end{proof}

\section{Concluding remarks and further questions}
We conclude by drawing attention to several questions which are open. We know that the complete graph is rigid. What about graphs which ``approximate" complete graphs e.g., expander graphs? Informally, an expander is a sequence of bounded degree graphs but with increasing number of vertices which has strong connectivity properties. Random graphs were known to be expanders since the works of Kolmogorov-Barzdin \cite{KB67} and also that of Pinsker \cite{Pin73} . The first explicit construction of an expander was given by Margulis in \cite{Mar82}.
To get an overview of how expanders approximate complete graphs, see Daniel A. Spielman's notes \cite{Spi15} or Oded Goldreich notes \cite{OG}.

\begin{Openproblem}
	Let $(\mathbf{\Gamma}_{n})_{n\in \mathbb{N}}$ denote a sequence of constant degree $d$ expander graphs. Is it true that almost every member of this sequence is a movable graph? 
\end{Openproblem}
\begin{remark}
	There exist constant degree expander graphs for which every member of the sequence is a movable graph e.g., in $SL_{2}(\mathbb{Z})$, take the Cayley graphs $$\mathbf{\Gamma}_{p} = C\Big(SL_{2}(\mathbb{Z}/p\mathbb{Z}), \Big\lbrace\begin{pmatrix}
	1 & 2\\
	0 & 1
	\end{pmatrix}^{\pm 1}, \begin{pmatrix}
	1 & 0\\
	2 & 1
	\end{pmatrix}^{\pm 1}\Big\rbrace\Big)$$ for $p$ running over the primes. These are the famous expanders of Margulis \cite{Mar82}. They are moving by Theorem \ref{DensestMovCay}. This phenomenon also occurs in  $SL_{n}(\mathbb{Z})$ for $n\geqslant 3$. For example, the graphs given by the Main Theorem in \cite{AB18} are moving by Theorem \ref{DensestMovCay}.
\end{remark}

The next question deals with movability and graph operations like squaring. By definition, the square graph $\mathbf{\Gamma}^{2}$ of an undirected graph $\mathbf{\Gamma} = (V,E)$ is the graph obtained by keeping the same vertex set $V$ while there is an edge $(u,w)$ in $\mathbf{\Gamma}^{2}$ iff there are two edges $(u,v)$, $(v,w)$ in $\mathbf{\Gamma}$ for some $v\in V$. 
\begin{Openproblem}\label{Q2}
	Is there any relationship between the rigidity, flexibility and movability of $\mathbf{\Gamma}$ and that of $\mathbf{\Gamma}^{2}$?\footnote{a slight modification is needed as we are discussing rigidity, flexibility and movability of loopless graphs whereas the square graph contains loops. When we pose this question, we mean the modified square graph with the loops removed.}
\end{Openproblem}

The above question and it's resolution is intrinsically related to the following: A method to construct infinitely many non-vertex transitive graphs which are moving. We know that the Cayley graph is vertex transitive. Adding an ``outer" edge will keep the movable property unchanged and make the graph non-vertex transitive. However, this is a sort of ad hoc construction of non-vertex transitive graphs. A better method will be to study undirected Cayley sum graphs. These are graphs defined on groups (equipped with a generating set for each group) such that there is an edge between $x$ and $y$ iff $x+y$ and $y+x$ belong to the generating set. Cayley sum graphs are non-vertex transitive in general. 

\begin{Openproblem}\label{Q3}
	What about conditions under which Cayley sum graphs are moving?
\end{Openproblem} 

Open Question \ref{Q3} is related to Open Question \ref{Q2} via the following correspondence: A Cayley sum graph $C_{\Sigma}(G,S)$ of a group $G$ is undirected iff the generating set $S$ is closed under conjugation. Under this condition, one can show that the square graph $\mathbf{\Gamma}^{2}$ of $\mathbf{\Gamma}= C_{\Sigma}(G,S)$ satisfies $\mathbf{\Gamma}^{2} = C(G,S^{-1}S) $ i.e., $\mathbf{\Gamma}^{2}$ is actually the undirected Cayley graph of $G$ with respect to the symmetric generating set $S^{-1}S$. We have a criterion for Cayley graphs to be flexible or movable. Thus any information about Open Question \ref{Q2} will give us information about Open Question \ref{Q3}.

\providecommand{\bysame}{\leavevmode\hbox to3em{\hrulefill}\thinspace}
\providecommand{\MR}{\relax\ifhmode\unskip\space\fi MR }
\providecommand{\MRhref}[2]{%
	\href{http://www.ams.org/mathscinet-getitem?mr=#1}{#2}
}
\providecommand{\href}[2]{#2}


\begin{thebibliography}{JJSS15}
	
	\bibitem[AB18]{AB18}
	Goulnara {Arzhantseva} and Arindam {Biswas}, \emph{{Large girth graphs with
			bounded diameter-by-girth ratio}}, arXiv e-prints (2018), arXiv:1803.09229.
	
	\bibitem[CD99]{Dix1899}
	A~C.~Dixon, \emph{On certain deformable frameworks}, Mess. Math. \textbf{29}
	(1899).
	
	\bibitem[FJK15]{FJK15}
	Zsolt Fekete, Tibor Jord\'{a}n, and Vikt\'{o}ria~E. Kaszanitzky, \emph{Rigid
		two-dimensional frameworks with two coincident points}, Graphs Combin.
	\textbf{31} (2015), no.~3, 585--599. \MR{3338020}
	
	\bibitem[GLS18]{GLS17}
	Georg Grasegger, Jan Legersk{\'y}, and Josef Schicho, \emph{Graphs with
		flexible labelings}, Discrete {\&} Computational Geometry (2018).
	
	\bibitem[GLS19]{GLS19}
	Georg Grasegger, Jan Legerský, and Josef Schicho, \emph{Graphs with flexible
		labelings allowing injective realizations}, Discrete Mathematics (2019),
	111713.
	
	\bibitem[Gol]{OG}
	Oded Goldreich, \emph{Basic facts about expander graphs},\\
	\url{http://www.wisdom.weizmann.ac.il/~oded/COL/expander.pdf}.
	
	\bibitem[GSS93]{GSS93}
	Jack Graver, Brigitte Servatius, and Herman Servatius, \emph{Combinatorial
		rigidity}, Graduate Studies in Mathematics, vol.~2, American Mathematical
	Society, Providence, RI, 1993. \MR{1251062}
	
	\bibitem[JJSS15]{JJSS15}
	Bill Jackson, Tibor Jord\'{a}n, Brigitte Servatius, and Herman Servatius,
	\emph{Henneberg moves on mechanisms}, Beitr. Algebra Geom. \textbf{56}
	(2015), no.~2, 587--591. \MR{3391192}
	
	\bibitem[KB67]{KB67}
	A.N. Kolmogorov and Y.M. Barzdin, \emph{On the realization of nets in 3-
		dimensional space}, Probl. Cybernet \textbf{2} (1967), no.~8, 261–268.
	
	\bibitem[Lam70]{Lam70}
	G.~Laman, \emph{On graphs and rigidity of plane skeletal structures}, Journal
	of Engineering Mathematics \textbf{4} (1970), no.~4, 331--340.
	
	\bibitem[Mar82]{Mar82}
	G.~A. Margulis, \emph{Explicit constructions of graphs without short cycles and
		low density codes}, Combinatorica \textbf{2} (1982), no.~1, 71--78.
	
	\bibitem[MT01]{MT01}
	H.~Maehara and N.~Tokushige, \emph{When does a planar bipartite framework admit
		a continuous deformation?}, Theoretical Computer Science \textbf{263} (2001),
	no.~1, 345 -- 354, Combinatorics and Computer Science.
	
	\bibitem[PG27]{PG27}
	H.~Pollaczek-Geiringer, \emph{Uber die gliederung ebener fachwerke.},
	Zeitschrift f\"ur Angewandte Mathematik und Mechanik (ZAMM) \textbf{7}
	(1927), no.~7, 58--72.
	
	\bibitem[Pin73]{Pin73}
	M.~Pinsker, \emph{On the complexity of a concentrator}, in 7th International
	Telegrafic Conference, pages 318/1318/4, 1973.
	
	\bibitem[Spi15]{Spi15}
	Daniel~A. Spielman, \emph{Properties of expander graphs},\\
	\url{http://www.cs.yale.edu/homes/spielman/561/lect15-15.pdf}, 2015.
	
	\bibitem[Sta14]{Sta14}
	Hellmuth Stachel, \emph{On the flexibility and symmetry of overconstrained
		mechanisms}, Philos. Trans. R. Soc. Lond. Ser. A Math. Phys. Eng. Sci.
	\textbf{372} (2014), no.~2008, 20120040, 15. \MR{3158341}
	
\end{thebibliography}

\end{document}